\documentclass[11pt]{article}
\usepackage{amsmath,amsfonts,amssymb,amsthm}
\usepackage{latexsym}
\usepackage{makeidx}
\usepackage[left=1.25in,right=1.25in,top=1.0in,bottom=1.25in]{geometry}

\newtheorem{definition}{Definition}
\newtheorem{theorem}[definition]{Theorem}
\newtheorem{lemma}[definition]{Lemma}
\newtheorem{fact}[definition]{Fact}

\newtheorem{claim}[definition]{Claim}

\newcommand\N{{\mathbb N}}

\newcommand{\restr}{\upharpoonright}

\date{}

\begin{document}
\title{Building suitable sets for locally compact groups by means of continuous selections\footnote{
{\em MSC Subj. Class.\/}: Primary: 22D05; Secondary: 22A05, 22C05, 54A25, 54B05, 54B35, 54C60, 54C65, 54D30, 54D45, 54H11.\endgraf
{\em Keywords and phrases\/}: locally compact group,  
generating rank, suitable set, lower semicontinuous, set-valued map, selection, converging sequence, compactly generated.
\endgraf
{\em Mailing address:\/} Division of Mathematics, Physics and Earth Sciences, 
Graduate School of Science and Engineering, 
Ehime University, 790-8577, Japan. 
{\em E-mail address:\/} {\tt dmitri@dpc.ehime-u.ac.jp\/}.
\endgraf
To appear in: {\sl Topology and its Applications\/}.  
\endgraf
{The author gratefully acknowledges partial financial support from the Grant-in-Aid for Scientific Research no.~19540092 by the Japan Society for the Promotion Science (JSPS).}
}}
\author{Dmitri Shakhmatov}
\maketitle
\begin{center}
{\sl Dedicated to the  memory of Jan Pelant\/}
\end{center}
\begin{abstract}
If a discrete subset $S$ of a topological group $G$ with the identity $1$ generates a dense subgroup of $G$ and $S\cup\{1\}$ is closed in $G$, then $S$ is called a {\em suitable set\/} for $G$.
We apply Michael's selection theorem to offer a direct, self-contained, purely topological proof of the result of Hofmann and Morris \cite{HM} on the existence of suitable sets in locally compact groups. Our approach 
uses only elementary facts from (topological) group theory. 
\end{abstract}

\def\grp#1{\left\langle{#1}\right\rangle}

All topological groups considered in this paper are assumed to be Hausdorff,
and all topological spaces are assumed to be Tychonoff.

\section{Motivating background}

Let $G$ be a group. We use $1_G$ to denote the identity element of $G$.
If $X$ is a subset of $G$, then $\grp{X}$ will denote the smallest subgroup of $G$ containing $X$, and we say that $X$ (algebraically) {\em generates\/} $\grp{X}$.

\begin{definition}
\label{suitable:set:definition}
{\rm \cite{D, mel, HM}}
{\rm 
A subset $X$ of a topological group $G$ is called a {\em suitable set\/} for $G$ provided that:
\begin{itemize}
\item[\rm{(i)}]
$X$ is discrete,
\item[\rm{(ii)}]
$X\cup\{1_G\}$ is closed in $G$, 
\item[\rm{(iii)}]
$\grp{X}$ is dense in $G$.
\end{itemize}
}
\end{definition}

Suitable sets were considered first in the early sixties by 
Tate in the framework of Galois cohomology (see \cite{D}). Tate 
proved\footnote{This proof is extremely condensed. Detailed proofs 
can be found in \cite{RZ} and \cite{HM2}.} that every profinite 
group has a suitable set. This result has later been proved also by Mel'nikov \cite{mel}. Later on,
Hofmann and Morris discovered the following fundamental theorem:
\begin{theorem}
\label{Hofmann-Morris:theorem}
{\rm \cite[Theorem 1.12]{HM}}
Every locally compact group has a suitable set.
\end{theorem}

Let us briefly outline main points of the proof 
from \cite{HM}.
The authors first prove the existence of suitable sets in compact connected Abelian groups. This is accomplished by using the full strength of the theory of free compact Abelian groups \cite{HM1}. The theorem for compact connected groups then follows from the Abelian compact connected case and the result of Kuranishi \cite{ku}
that every compact connected simple group has a dense subgroup generated by two elements. 
Since compact totally disconnected groups have suitable sets 
by the results of Tate and Mel'nikov (cited above), the authors of \cite{HM}
then combine connected and totally disconnected cases together to get the 
conclusion   
for all compact groups by deploying a theorem of Lee \cite{Lee}: Every compact group $G$ contains a closed totally disconnected subgroup $K$ such that
$G=c(G)\cdot K$, where $c(G)$ is the connected component of $G$. Having proved the result in compact case, Hofmann and 
Morris then proceed to deduce the general case from the compact case 
using some structure theorems for locally compact groups.

The main purpose of this article is to offer a direct, self-contained, 
purely topological proof of Theorem \ref{Hofmann-Morris:theorem} 
based on Michael's selection theorem. Our proof is in the spirit of \cite{Uspenskii, Sh}, and uses only elementary facts from (topological) group theory. 

Theorem \ref{Hofmann-Morris:theorem} allowed 
Hofmann and Morris \cite{HM} to introduce the {\em generating rank\/} 
$$
s(G)=\min\{|X|: X \text{ is a suitable set for }G\}
$$
of a locally compact group $G$. (For profinite groups, $s(G)$ has been already defined by Mel$'$nikov \cite{mel}.)
As witnessed by the fact that the whole Chapter 12 of the monograph \cite{HMbook} by Hofmann and 
Morris is devoted to the study of this cardinal function
(and its relation to the weight),
$s(G)$ is undoubtedly one of the most 
important cardinal invariants of a (locally) compact group $G$. 

Let $G$ be a topological group.
Following \cite{DS} define
the {\em topologically generating weight $tgw(G)$\/} of $G$ by
$$
tgw(G)=\min\{w(F): F \text{ is closed in } G \text{ and }\grp{F} \text { is 
dense in }
G\}, 
$$ 
where $w(X)=\min\{|\mathcal{B}|: \mathcal{B}$ is a base of $X\}+\omega$
is the {\em weight\/} of a space $X$.
The two principle results of \cite{DS} are summarized in the following
\begin{theorem}
\label{TGW:theorem}
Let $G$ be a compact group. Then:
\begin{itemize}
\item[(i)]
$tgw(G)=s(G)$ whenever $s(G)$ is infinite, and
\item[(ii)]
$tgw(G)=w(G/c(G))\cdot \sqrt[\omega]{w(c(G))}$,
where $c(G)$ is the connected component of $G$ and $\sqrt[\omega]{\tau}$ is defined to be the smallest infinite cardinal
$\kappa$ such that $\kappa^\omega\ge\tau$.
\end{itemize}
\end{theorem} 
The proof of this theorem in \cite{DS} is essentially topological and completely self-contained {\em with the only exception\/} of Theorem \ref{Hofmann-Morris:theorem} which is still necessary.
Our present manuscript 
completes the job started in \cite{DS}
by providing a self-contained,
purely topological proof of Theorem \ref{Hofmann-Morris:theorem}.
It is worth mentioning that in \cite[Section 9]{DS} Theorem \ref{TGW:theorem} has been used to deduce
(as straightforward corollaries) a series of major results from  Chapter 12 of the monograph \cite{HMbook}
by Hofmann and 
Morris.

\section{Necessary facts}

\def\lsc{lower semicontinuous}

In this section we collect (mostly) 
well-known facts that will be used in the proof.

Recall that a map $f:X\to Y$ is:

(i)
{\em open\/} provided that $f(U)$ is open 
in $Y$ for every open subset $U$ of $X$,

(ii)
 {\em closed\/} provided that $f(F)$ is closed 
in $Y$ for every closed subset $F$ of $X$,

(iii)
{\em perfect\/} if $f$ is a closed map and $f^{-1}(y)$ is compact for every $y\in Y$.

\begin{fact}
\label{perfect:is:divisible}
{\rm \cite[Proposition 3.7.5]{Engelking}}
Assume that $f:X\to Y$ and $g:Y\to Z$ are continuous surjections and the map
$g\circ f:X\to Z$ is perfect. Then $g$ is also perfect.
\end{fact}

\def\diag{\triangle}
For every $i\in I$ let $f_i:X\to Y_i$ be a map. The {\em diagonal product\/}
$\diag\{f_i:i\in I\}$ of the family $\{f_i:i\in I\}$ is a map 
$f:X\to \prod\{Y_i:i\in I\}$ which assigns
to every $x\in X$ the point $\{f_i(x)\}_{i\in I}$ of the Cartesian product 
$\prod\{Y_i:i\in I\}$. 
(More precisely, $f$ assigns to each $x\in X$ the point 
$f(x)\in \prod\{Y_i:i\in I\}$ defined by $f(x)(i)=f_i(x)$ for all $i\in I$.)

\begin{fact}
\label{diagonal:of:perfect:maps}
{\rm \cite[Theorem 3.7.10]{Engelking}}
For every $i\in I$ let $f_i:X\to Y_i$ be a continuous perfect map.
Then the diagonal product $\diag\{f_i:i\in I\}$ is also a continuous 
perfect map.
\end{fact}

\begin{fact}
\label{quotients:with:respect:to:compact:group:are:perfect}
{\rm \cite[Chapter II, Theorem 5.18]{HR}}
If $N$ is a compact normal subgroup of a topological group $G$, then 
the quotient map from $G$ onto its quotient group $G/N$ is perfect.
\end{fact}

\begin{fact}
\label{quotient:implies:open}
Let $\pi:G\to H$ be a continuous group homomorphism  
from a topological group $G$ onto a 
topological group $H$.
If $\pi$  
is a quotient map, then $\pi$ is also an open map.
In particular, if $\pi$ is a perfect map, then $\pi$ is an open map.
\end{fact}
\begin{proof}
The first statement follows from \cite[Chapter II,
Theorem 5.17]{HR}. To prove the second statement note that  
a prefect map is a closed map, and every closed map is a quotient map
\cite[Corollary 2.4.8]{Engelking}.
\end{proof}

\begin{fact}
\label{locally:compact:subgroups:are:closed}
{\rm \cite[Chapter II, Theorem 5.11]{HR}}
A locally compact subgroup $G$ of a topological group $H$ is closed in $H$. 
\end{fact}

Recall that a topological group $G$ is {\em compactly generated\/} provided that there exists a compact subset $K$ of $G$ such that $G=\grp{K}$.
\begin{fact}
\label{separable:metric:quotient}
{\rm \cite{KakutaniKodaira}}
If $U$ is an open subset of a compactly generated,
locally compact group $G$, then there exists a compact
normal subgroup $N\subseteq U$ of $G$ such that $G/N$ has
a countable base.
\end{fact}
We note that in \cite[Chapter II, Theorem 8.7]{HR} one finds a purely 
topological, elementary proof of Fact \ref{separable:metric:quotient} that does not use the structure theory of locally compact groups.

\begin{definition}
{\rm 
If $D$ is an infinite set, then $S(D)=D\cup\{*\}$ will denote 
the one-point compactification of the discrete set of size $|D|$. 
(Here $*\not\in D$.) 
That is, 
all points of $D$ are isolated in $S(D)$, and the family 
$\{S(D)\setminus F: F$ is a finite subset of $D\}$ consists of 
open neighbourhoods of a single non-isolated point $*$. 
}
\end{definition}

Note that $S(D)$ can be characterized as a compact Hausdorff space of size $|D|$
having precisely one non-isolated point.
The relevance of this space 
to our topic can be seen from the following folklore 
\begin{fact}
\label{fact11}
If $X$ is an infinite suitable set for a compact group $G$, then 
the subspace $X\cup\{1_G\}$ 
of $G$ 
is compact and homeomorphic
to the space $S(X)$. 
\end{fact}
\begin{proof}
Indeed, 
$X\cup \{1_G\}$ is closed in $G$
by item (ii) of Definition \ref{suitable:set:definition}.
Since $G$ is compact, so is $X\cup \{1_G\}$. 
Since $X$ is an infinite discrete subset of $G$ by item (i) of
Definition \ref{suitable:set:definition}, 
the point $1_G$ cannot be isolated in $X\cup \{1_G\}$
(otherwise $X\cup \{1_G\}$ would become an infinite discrete compact space).
Hence, $X\cup \{1_G\}$ is a compact space with a single non-isolated point
$1_G$, and thus $X\cup \{1_G\}$ is homeomorphic to $S(X)$.
\end{proof}

\begin{fact}
\label{fact12}
Assume that $X$ is a compact space with a single non-isolated point $x$ and 
$f:X\to Y$ is a continuous surjection of $X$ onto an infinite space $Y$. Then $Y$ is a compact space with a single non-isolated point $f(x)$. 
\end{fact}
\begin{proof} 
We are going to show first that $Y\setminus V$ is finite for every
open subset $V$ of $Y$ containing $f(x)$. 
Indeed,
since $f:X\to Y$ is continuous, $U=f^{-1}(V)$ is an open subset of $X$ containing $x$. 
Since every point of $X$ different from $x$ is isolated, $X\setminus U$ consists of isolated points of $X$. Since $X$ is compact, we conclude that the set $X\setminus U$ is finite. Therefore, the set $Y\setminus V$ must be finite 
as well.
Since $Y$ is an infinite set, $V$ must be infinite. Thus, $f(x)$ is a non-isolated point of $Y$. 

Let us show next that $Y$ is compact. Let $\mathcal{V}$ be an open cover of $Y$. There exists $V\in\mathcal{V}$ such that $f(x)\in V$. For every $y\in Y\setminus V$ choose $V_y\in \mathcal{V}$ with $y\in V_y$. Now $\{V_y:y\in Y\setminus V\}\cup\{V\}$ is a finite subcover of $\mathcal{V}$.

Finally, let $y\in Y\setminus \{f(x)\}$. Since $Y$ is Hausdorff, there exist
open subsets $W$ and $V$ of $Y$ such that $y\in W$, $f(x)\in V$ and $W\cap V=\emptyset$. Then $W\subseteq Y\setminus V$, and hence $W$ is finite. Since every singleton is a closed subset of $Y$, it now follows that $y$ is an isolated point of $Y$.
\end{proof}

Our next lemma,  which is in a certain sense the ``converse'' of 
Fact \ref{fact11}, is  the key to building suitable sets in (compact-like) topological groups. 
\begin{lemma}
\label{building:suitable:sets}
Suppose that $G$ is a topological group, $X$ is an infinite set and 
$f:S(X)\to G$ is a continuous map such that $f(*)=1_G$
and $\grp{f(S(X))}$ is dense in $G$. Then 
$S=f(S(X))\setminus \{1_G\}$ is a suitable set for $G$ such that
$S\cup\{1_G\}$ is compact. 
\end{lemma}
\begin{proof}
Suppose first that $f(S(X))$ is a finite set. Then $S$ is discrete,
$S\cup\{1_G\}$ is compact and closed (being finite), and 
$\grp{S}=\grp{S\cup\{1_G\}}=\grp{f(S(X))}$ is dense in $G$.
Therefore, $S$ is a suitable set for $G$.

Assume now that $f(S(X))$ is infinite.
As an infinite
 continuous image of the compact space $S(X)$ with a single non-isolated 
point $*$, the space $f(S(X))$ is also a compact space with a single non-isolated point $f(*)=1_{G}$ (Fact \ref{fact12}). 
Therefore, $S=f(S(X))\setminus \{1_{G}\}$
is a discrete set and $S\cup \{1_{G}\}$ is compact (and thus closed in $G$).
Moreover, 
$\grp{S}=\grp{f(S(X)\setminus \{1_{G}\})}=\grp{f(S(X))}$.
Since the latter set  is dense in $G$,
we conclude that $S$ is a suitable set for $G$.
\end{proof}

Note that $S(\mathbb{N})$ is (homeomorphic to) a non-trivial convergence sequence together with its limit. 
The next fact is a key ingredient in our proof, so to make our manuscript self-contained we include its proof adapted from \cite{FSh}.
\begin{fact}
\label{Fujita:Sh}
{\rm \cite{FSh}}
Let $G$ be a compactly generated 
metric group.
Then there exists a continuous map $f:S(\mathbb{N})\to G$ such that 
$f(*)=1_G$ and $\grp{f(S(\mathbb{N}))}$ is dense in $G$.
\end{fact}
\begin{proof}
Fix a local base $\{V_n:n\in\N\}$ at $1_G$ such that
$V_0=G$ and $V_{n+1}\subseteq V_n$ for all $n\in\mathbb{N}$.
Let $G=\grp{K}$, where $K$ is a compact subset of $G$.
One can easily see that $G$ is separable, so 
let $D=\{d_n:n\in\N\}$ be a countable dense subset of $G$.

Fix $n\in\mathbb{N}$. Since 
$\{xV_{n+1}:x\in G\}$ is an open cover of $G$ and $K$ is a compact subset of $G$,
$K\subseteq \bigcup\{xV_{n+1}:x\in F_n\}$ for some finite set $F_n$.
Now we have 
\begin{equation}
\label{EQ:1}
G=\grp{K}\subseteq \grp{\bigcup\{xV_{n+1}:x\in F_n\}}\subseteq 
\grp{F_n\cup V_{n+1}}.
\end{equation}

By induction on $n$ we will define 
a sequence $\{E_n:n\in\N\}$ of finite subsets 
of $G$ with the following properties:

(i$_n$) $E_n\subseteq V_n$,

(ii$_n$) $G\subseteq\langle E_0\cup E_1\cup\cdots\cup E_{n}\cup V_{n+1}\rangle$, and

(iii$_n$) $d_{n}\in\langle E_0\cup E_1\cup\cdots\cup E_{n}\rangle$.  

To begin with, note that the set $E_0=F_0\cup\{d_0\}$ satisfies all
three conditions (i$_0$)--(iii$_0$).  
Suppose that we have already defined 
finite sets $E_0$,$E_1$,\dots, $E_{n-1}$ such 
that conditions (i$_0$), $\dots$, (i$_{n-1}$),
(ii$_0$), $\dots$, (ii$_{n-1}$) and
(iii$_0$), $\dots$, (iii$_{n-1}$) are satisfied.
Condition (ii$_{n-1}$) implies that
$$
F_n\cup\{d_n\}\subseteq \langle E_0\cup E_1\cup\cdots\cup E_{n-1}\cup V_{n}\rangle,
$$
and since $F_n$ is finite,
we can find a finite set $E_{n}\subseteq V_n$ 
such that 
\begin{equation}
\label{EQ:2}
F_n\cup\{d_n\}\subseteq
\grp{E_0\cup E_1\cup \cdots\cup E_{n-1}\cup E_n}.  
\end{equation}
Conditions (i$_n$) and (iii$_n$) are clear, and 
(ii$_n$) follows from (\ref{EQ:1}) and (\ref{EQ:2}).

From
(i$_n$) for $n\in\mathbb{N}$ 
it follows that the set 
$S=\bigcup\{E_n:n\in\N\}$ forms a sequence converging to $1_G$. 
Since (iii$_n$) holds for every $n\in\mathbb{N}$,
we get $D\subseteq \grp{S}$, and so $\grp{S}$ is dense in $G$.
Now take any bijection $f:\mathbb{N}\to S$ and define also $f(*)=1_G$. 
\end{proof}

Recall that a {\it set-valued map\/} is a map
$F:Y\to Z$ which assigns to every point
$y\in Y$ a non-empty closed subset $F(y)$ of $Z$. This set-valued map is
{\it \lsc\/}
if
$V=\{y\in Y:F(y)\cap U\neq\emptyset\}$
is open in $Y$ for every open subset $U$ of $Z$.
A (single-valued)
map $f:Y\to Z$ is called a {\it selection of $F$\/}
provided that 
$f(y)\in F(y)$ for all $y\in Y$.

We finish this section with the following special case of
Michael's selection theorem \cite[Theorem 2]{Michael} (see also \cite{RS}).

\begin{fact}
\label{michael:selection:theorem}
A \lsc\ set-valued map $F:Y\to Z$ from a zero-dimensional
(para)com\-pact space $Y$ 
to a complete metric space $Z$ has a continuous selection $f:Y\to Z$.
\end{fact}

\section{Lifting lemmas based on Michael's selection theorem}

\begin{lemma}
\label{selection:lemma}
Suppose that $K_0, K_1$ are topological groups, 
$N$ is a subgroup of the product $K_0\times K_1$,
and for each $i=0,1$ let $q_i=p_i\restr_{N}:N\to K_i$ be the restriction 
to $N$ of 
the projection $p_i:K_0\times K_1\to K_i$ onto the $i$th coordinate.
Assume also that:
\begin{itemize}
\item[\rm{(1)}]
$K_i=p_i(N)$ for each $i=0,1$,
\item[\rm{(2)}]
$q_0$ is an open map,
\item[\rm{(3)}]
$q_1$ is a closed map,
\item[\rm{(4)}]
$K_1$ is a compete metric space,
\item[\rm{(5)}]
$Y$ is a (para)compact zero-dimensional space and $h:Y\to K_0$ is a continuous map.
\end{itemize}
Then there exists a continuous map $g:Y\to N$ such that
$h=q_0\circ g$.
\end{lemma}

\begin{proof}
For $y\in Y$ define
$F(y)=\{z\in K_1: (h(y),z)\in N\}$.
Note that $N\cap(\{h(y)\}\times K_1)$ is a closed subset of $N$, and 
so the set $F(y)=q_1(N\cap(\{h(y)\}\times K_1))$ 
must be closed in $q_1(N)=p_1(N)=K_1$ by (1) and (3).
Since $h(y)\in K_0=p_0(N)$ by (1) and (5), it follows that
$F(y)\neq\emptyset$.
Therefore $F:Y\to K_1$ is a set-valued map. 

We claim that $F$ is \lsc.
Indeed, let $U$ be an open subset of $K_1$.
Since 
$N\cap (K_0\times U)$ is an open subset of $N$, 
$q_0(N\cap (K_0\times U))$ is an open subset of $q_0(N)=p_0(N)=K_0$
by (1) and (2).
Since $h:Y\to K_0$ is a continuous map by (5),
$V=h^{-1}(q_0(N\cap (K_0\times U)))$ is an open subset of $Y$. 
Now note that $V=\{y\in Y:F(y)\cap U\neq\emptyset\}$ by definitions of $F$ and $V$. 

In view of (4), the assumptions of Fact \ref{michael:selection:theorem}
are satisfied if one takes $K_1$ as $Z$.
Let $f:Y\to K_1$ be a (single-valued) continuous selection of $F$
which exists by the conclusion of Fact \ref{michael:selection:theorem}.

Define $g:Y\to K_0\times K_1$ by $g(y)=(h(y),f(y))$ for $y\in Y$.
Since both $h$ and $f$ are continuous, so is $g$.
If
$y\in Y$, then 
$g(y)=(h(y),f(y))\in \{h(y)\}\times F(y)$ because $f$ is a selection of $F$,
which yields 
$g(y)\in N$ by the definition of $F(y)$. Therefore, $g(Y)\subseteq N$.
The equality $h=q_0\circ g$ is obvious from our definition of $g$.
\end{proof}

In the sequel we will only need a
particular case when the previous lemma is applicable:

\begin{lemma}
\label{main:lemma}
Suppose that $G$ is a locally compact group, 
$K_0$ is a topological group, $K_1$ is a metric group,
$\chi_i:G\to K_i$ is a continuous group homomorphism for $i=0,1$,
$\chi=\chi_0\diag\chi_1:G\to K_0\times K_1$ is the diagonal product of 
maps $\chi_0$ and $\chi_1$, and $N=\chi(G)$.
Assume also that:
\begin{itemize}
\item[\rm{(a)}]
$K_i=\chi_i(G)$ for each $i=0,1$,
\item[\rm{(b)}]
each $\chi_i$ is a perfect map,
\item[\rm{(c)}]
$Y$ is a (para)compact zero-dimensional space and $h:Y\to K_0$ is 
a continuous map.
\end{itemize}
Then there exists a continuous map $g:Y\to N$ such that
$h=q_0\circ g$, where 
$q_0=p_0\restr_{N}:N\to K_0$ is the restriction 
to $N$ of 
the projection $p_0:K_0\times K_1\to K_0$.
\end{lemma}
\begin{proof}
It suffices to check that $N$, $Y$ and $h$ 
satisfy all the assumptions of Lemma
\ref{selection:lemma}. (1) follows from (a).
Let $i=0,1$.
Since both $\chi:G\to N$ and $q_i:N\to K_i$ are surjections,  
$\chi_i=q_i\circ \chi$ and $\chi_i$ is a perfect map by item (b),
$q_i$ is a perfect map
(Fact \ref{perfect:is:divisible}), and so also an open map 
(Fact \ref{quotient:implies:open}).
This yields both (2) and (3).
Being an open continuous image of a locally compact space $G$, $K_1$ is locally compact. Since a locally compact metric space 
admits a complete metric, we get (4).
Finally, (5) coincides with (c). Now the conclusion of our lemma follows 
from the conclusion of Lemma \ref{selection:lemma}.
\end{proof}

\section{Proof of Theorem \ref{Hofmann-Morris:theorem}}

If $G$ and $H$ are groups and $f:G\to H$ is a group homomorphism, then 
$\ker f=\{x\in G:f(x)=1_H\}$ denotes the {\em kernel\/} of $f$. 
Obviously, $\ker f$ is a normal subgroup of $G$.

We are now ready to prove a specific version of Theorem \ref{Hofmann-Morris:theorem}.
Our proof is based on representing a compactly generated, locally compact
group as a limit of some inverse spectra (aka a projective limit in the terminology of algebraists) of 
locally compact separable metric groups. In order to make an exposition easier to comprehend for readers not familiar with inverse (aka projective) limits, we have chosen the presentation using diagonal products of maps, thereby allowing for a much simpler visualiziation of such a limit.

\begin{theorem}
\label{main:theorem}
Let $G$ be a topological group generated by its open subset with 
compact
 closure.
Then $G$ has a suitable set $S$ such that $S\cup\{1_G\}$ is compact.  
\end{theorem}
\begin{proof}
Fix a local base $\{U_\alpha:\alpha<\tau\}$
at $1_G$. 
If $\tau\le \omega$, then $G$ is a compactly generated 
metric group,
and hence $G$ has the desired suitable set by Fact \ref{Fujita:Sh}
and Lemma \ref{building:suitable:sets}.

From now on we will assume that $\tau\ge\omega_1$.
Let $X$ be a set with $|X|=\tau$.
For every ordinal
$\alpha<\tau$, apply Fact \ref{separable:metric:quotient} to choose a compact normal subgroup $N_\alpha$ of $G$ such that
$N_\alpha\subseteq U_\alpha$ and $H_\alpha=G/N_\alpha$ has a countable base, and let 
$\psi_\alpha: G\to H_\alpha$ be the quotient map. For every ordinal $\alpha$ 
satisfying $1\le \alpha\le \tau$ define 
$\varphi_\alpha=\diag\{\psi_\beta:\beta<\alpha\}: G\to\prod \{H_\beta:\beta<\alpha\}$
and
$G_\alpha=\varphi_\alpha(G)$.
For $1\le\beta\le \alpha\le\tau$ let
$\varpi^\alpha_\beta: \prod \{H_\gamma:\gamma<\alpha\}\to\prod \{H_\gamma:\gamma<\beta\}$
be the natural projection, and define 
$\pi^\alpha_\beta=\varpi^\alpha_\beta\restr_{G_\alpha}: G_\alpha\to G_\beta$ to be the restriction of $\varpi^\alpha_\beta$ to $G_\alpha\subseteq \prod \{H_\gamma:\gamma<\alpha\}$. Note that $\pi^\alpha_\beta$ is a surjection.
By our construction, 
\begin{equation}
\label{eq:0}
\varphi_\alpha\circ\pi^\alpha_\beta=\varphi_\beta 
\text{ and }
\pi^\alpha_\gamma=\pi^\beta_\gamma\circ \pi^\alpha_\beta
\text{ whenever } 
1\le\gamma\le \beta\le \alpha\le\tau.
\end{equation}
\begin{claim}
\label{phi:are:perfect}
$\varphi_\alpha$ is a perfect map for every $\alpha$ with $1\le\alpha\le\tau$.
\end{claim}
\begin{proof}
Each $\psi_\beta$ is a perfect map by Fact \ref{quotients:with:respect:to:compact:group:are:perfect}, so 
the map
$\varphi_\alpha=\diag\{\psi_\beta:\beta<\alpha\}$ is also perfect by 
Fact \ref{diagonal:of:perfect:maps}.
\end{proof}

By transfinite recursion on $\alpha$, for every ordinal 
$\alpha$ satisfying $1\le \alpha\le\tau$ we will define a continuous map 
$f_\alpha:S(X)\to G_\alpha$ satisfying the following properties:
\begin{itemize}
\item[(i$_\alpha$)] 
$f_\beta=\pi^\alpha_\beta\circ f_\alpha$ whenever $1\le \beta<\alpha$,
\item[(ii$_\alpha$)]
$f_\alpha(*)=1_{G_\alpha}$,
\item[(iii$_\alpha$)]
$|\{x\in X:f_\alpha(x)\neq 1_{G_\alpha}\}|\le\omega\cdot|\alpha|$,
\item[(iv$_\alpha$)]
$\grp{f_\alpha(S(X))}$ is dense in $G_\alpha$.
\end{itemize}

To motivate these conditions, we mention that (ii$_\alpha$) and (iv$_\alpha$)
guarantee that $f_\alpha(S(X))\setminus \{1_{G_\alpha}\}$ is a suitable set 
for $G_\alpha$ (Lemma \ref{building:suitable:sets}). The other two conditions
(i$_\alpha$) and (iii$_\alpha$) are technical and needed only for carrying out
the recursion construction.

We start our recursion with $\alpha=1$.
First of all note that $\varphi_1=\psi_0$ and $G_1=H_0$.
Being a continuous homomorphic image of a compactly generated group $G$,
$G_1$ itself is compactly generated.
Let $N$ be a countable subset of $X$. Since $S(N)$ and $S(\mathbb{N})$ are 
homeomorphic, applying Fact \ref{Fujita:Sh} we can find a continuous  
map $f:S(N)\to G_1$ such that 
$f(*)=1_{G_1}$ and $\grp{f(S(N))}$ is dense in $G_1$.
We extend this map to the continuous map $f_1:S(X)\to G_1$
by defining $f_1(x)=1_{G_1}$ for every $x\in X\setminus N$
and $f_1(y)=f(y)$ for $y\in S(N)$. Now note that
$f_1$ satisfies properties (i$_1$)--(iv$_1$).

Suppose now that 
$\alpha$ is an ordinal with $1<\alpha\le\tau$. Assume also that 
a continuous map $f_\beta:S(X)\to G_\beta$ satisfying properties (i$_\beta$)--(iv$_\beta$) has been already defined for every ordinal 
$\beta$ such that $1\le\beta<\alpha$. We are going to define 
a continuous map 
$f_\alpha:S(X)\to G_\alpha$ satisfying properties (i$_\alpha$)--(iv$_\alpha$).
As usual, we consider two cases.

\medskip
{\it Case 1.\/} {\sl $\alpha=\beta+1$ is a successor ordinal\/}. 
Clearly, a subspace 
\begin{equation}
\label{Y:beta}
Y_\beta=\{x\in X:f_\beta(x)\neq 1_{G_\beta}\}\cup \{*\}
\end{equation} 
of $S(X)$ is closed in $S(X)$. Hence, $Y_\beta$ is a compact space 
with at most 
one non-isolated point. In particular, $Y_\beta$ is zero-dimensional. 

We claim that $K_0=G_\beta$, $K_1=H_\beta$, $\chi_0=\varphi_\beta$,
$\chi_1=\psi_\beta$, $N=G_\alpha$, $Y=Y_\beta$ and 
$h=f_\beta\restr_{Y_\beta}$ satisfy the assumptions of Lemma
\ref{main:lemma}. Indeed, $\chi=\chi_0\diag\chi_1=\varphi_\beta\diag\psi_\beta=\varphi_\alpha$, and so $N=G_\alpha=\varphi_\alpha(G)=\chi(G)$.
(a) holds trivially. 
The map $\chi_0=\varphi_\beta$ is perfect by Claim \ref{phi:are:perfect}, while 
$\chi_1=\psi_\beta$ is a perfect map by Fact 
\ref{quotients:with:respect:to:compact:group:are:perfect}.
This proves (b). Since $f_\beta$ is a continuous map, so is $h=f_\beta\restr_{Y_\beta}$.
Thus establishes (c).

Let $g:Y_\beta\to G_\alpha$ be a continuous map satisfying 
$f_\beta\restr_{Y_\beta}=\pi^\alpha_\beta\circ g$
which exists according to the conclusion 
of Lemma
\ref{main:lemma}.
Define $g':Y_\beta\to G_\alpha$ by $g'(y)=g(y)\cdot g(*)^{-1}$ for 
$y\in Y_\beta$. Clearly, $g'$ is a continuous map and $g'(*)=1_{G_\alpha}$.
If $y\in Y_\beta$, then 
$\pi^\alpha_\beta\circ g(*)=f_\beta\restr_{Y_\beta}(*)=f_\beta(*)=1_{G_\beta}$
by (ii$_\beta$), and so
$$
\pi^\alpha_\beta(g'(y))=\pi^\alpha_\beta(g(y)\cdot g(*)^{-1})=
\pi^\alpha_\beta(g(y))\cdot \pi^\alpha_\beta(g(*))^{-1}=
f_\beta\restr_{Y_\beta}(y)\cdot (1_{G_\beta})^{-1}=f_\beta\restr_{Y_\beta}(y)
$$
because $\pi^\alpha_\beta$ is a group homomorphism.
This gives 
\begin{equation}
\label{g:prime}
\pi^\alpha_\beta\circ g'=f_\beta\restr_{Y_\beta}.
\end{equation}

Since $\beta\ge 1$, from (\ref{eq:0}) we have 
$\ker \pi^\alpha_\beta= \varphi_\alpha(\ker\varphi_\beta)
\subseteq \varphi_\alpha(\ker \psi_0)\subseteq 
\varphi_\alpha(N_0)$. 
Since $N_0$ is compact, so is $\varphi_\alpha(N_0)$. 
Being a closed subspace of $\varphi_\alpha(N_0)$, $\ker \pi^\alpha_\beta$ must be 
compact. Since $\ker \pi^\alpha_\beta\subseteq \{1_{G_\beta}\}\times H_\beta$ and $H_\beta$ has a countable base, $\ker \pi^\alpha_\beta$ is a compact metric group. 

Note that $|Y_\beta|\le\omega\cdot|\beta|<\tau$ by (iii$_\beta$), 
and since $\tau\ge\omega_1$,
we can choose
a countable set $Z_\beta\subseteq X$ with $Y_\beta\cap Z_\beta=\emptyset$.
Since $Z_\beta\cup \{*\}$ is naturally homeomorphic to 
$S(\mathbb{N})$, Fact \ref{Fujita:Sh} allows us to find a continuous map
$\theta:Z_\beta\cup \{*\}\to \ker\pi^\alpha_\beta\subseteq G_\alpha$ such that
$\theta(*)=1_{G_\alpha}$ and $\grp{\theta(Z_\beta)}$ is dense in $\ker\pi^\alpha_\beta$.

Now define the map $f_\alpha:S(X)\to G_\alpha$ by 
$$
f_\alpha(x)=
\left\{\begin{array}{ll}
g'(x) & \text{ if } x\in Y_\beta,\\
\theta(x) & \text{ if } x\in Z_\beta,\\
1_{G_\alpha} & \text{ if } x\in S(X)\setminus (Y_\beta\cup Z_\beta).
\end{array}\right.
$$
Since both $g'$ and $\theta$ are continuous maps, one can easily check that 
the map $f_\alpha$ is continuous as well.

\begin{claim}
\label{alpha:beta}
$f_\beta=\pi^\alpha_\beta\circ f_\alpha$.
\end{claim}
\begin{proof}
 If $y\in Y_\beta$, then 
$\pi^\alpha_\beta(f_\alpha(y))=\pi^\alpha_\beta(g'(y))=f_\beta\restr_{Y_\beta}(y)=f_\beta(y)$
by (\ref{g:prime}).

Suppose now that $x\in S(X)\setminus Y_\beta$.
We claim that $\pi^\alpha_\beta(x)=1_{G_\beta}$. Indeed,
if $x\in Z_\beta$, then 
$\pi^\alpha_\beta(f_\alpha(x))=\pi^\alpha_\beta(\theta(x))=1_{G_\beta}$
because $\theta(x)\in\theta(Z_\beta)\subseteq \ker\pi^\alpha_\beta$.
If $x\in S(X)\setminus (Y_\beta\cup Z_\beta)$, then $f_\alpha(x)=1_{G_\alpha}$, and so
$\pi^\alpha_\beta(f_\alpha(x))=\pi^\alpha_\beta(1_{G_\alpha})=1_{G_\beta}$.
Finally, (\ref{Y:beta}) and (ii$_\beta$) yields  
$f_\beta(x)=1_{G_\beta}
=\pi^\alpha_\beta(f_\alpha(x))$ 
for 
$x\in S(X)\setminus Y_\beta$. 
\end{proof}

Let us check now conditions (i$_\alpha$)--(iv$_\alpha$).

(i$_\alpha$) 
Suppose that $1\le \gamma<\alpha=\beta+1$. If $\gamma=\beta$,
then Claim \ref{alpha:beta} applies.
Suppose now that $1\le\gamma<\beta$.
Then 
$\pi^\alpha_\gamma\circ f_\alpha=\pi^\beta_\gamma\circ \pi^\alpha_\beta\circ f_\alpha=
\pi^\beta_\gamma\circ f_\beta=f_\gamma$ by 
(\ref{eq:0}),
Claim \ref{alpha:beta}
 and (i$_\beta$).

(ii$_\alpha$) $f_\alpha(*)=g'(*)=1_{G_\alpha}$.

(iii$_\alpha$)
From the definition of $f_\alpha$ one has
 $\{x\in X:f_\alpha(x)\neq 1_{G_\alpha}\}\subseteq Y_\beta\cup Z_\beta$,
and so
$|\{x\in X:f_\alpha(x)\neq 1_{G_\alpha}\}|\le|Y_\beta|\cdot |Z_\beta|
\le 
\omega\cdot|\beta|\cdot \omega\le \omega\cdot |\alpha|$
by (iii$_\beta$).

(iv$_\alpha$)
Let $F$ be the closure of $\grp{f_\alpha(S(X))}$ in $G_\alpha$. We need to show that $F=G_\alpha$. Observe that
$\grp{f_\alpha(Z_\beta)}\subseteq \grp{f_\alpha(S(X))}\subseteq F$.
Since $\grp{f_\alpha(Z_\beta)}$ is dense in $\ker \pi^\alpha_\beta$, it now follows that $\ker \pi^\alpha_\beta\subseteq F$.
Since both $\varphi_\alpha$ and $\pi^\alpha_\beta$ are surjections and 
$\pi^\alpha_\beta\circ \varphi_\alpha=\varphi_\beta$ is a perfect map
by (\ref{eq:0}) and Claim \ref{phi:are:perfect}, Fact \ref{perfect:is:divisible}
allows us to conclude that
$\pi^\alpha_\beta$ is a perfect (and hence also closed) map. Therefore,
$\pi^\alpha_\beta(F)$ is a closed subset of $G_\beta$.
From  (i$_\alpha$) one gets
$\pi^\alpha_\beta({f_\alpha(S(X))})=f_\beta(S(X))$, 
and since $\pi^\alpha_\beta$ is a group homomorphism, 
one also has
$\grp{f_\beta(S(X))}=\pi^\alpha_\beta(\grp{f_\alpha(S(X))})\subseteq \pi^\alpha_\beta(F)$.
According to (iv$_\beta$), the set $\grp{f_\beta(S(X))}$ is dense in 
$G_\beta$, and since $\pi^\alpha_\beta(F)$ is closed in $G_\beta$, 
this yields $\pi^\alpha_\beta(F)=G_\beta$.
Since $F$ is a subgroup of $G_\alpha$ satisfying both $\ker\pi^\alpha_\beta\subseteq F$ and $\pi^\alpha_\beta(F)=G_\beta=\pi^\alpha_\beta(G_\alpha)$, one obtains
$F=G_\alpha$.

\medskip
{\it Case 2.\/} {\sl $\alpha$ is a limit ordinal\/}. 
Define 
\begin{equation}
\label{def:L}
L_\alpha=\left\{h\in\prod\{H_\beta:\beta<\alpha\}:h\restr_\beta\in G_\beta
 \text{ whenever } 1\le\beta<\alpha\right\}.
\end{equation}

\begin{claim}
\label{dense:claim}
Suppose that $H\subseteq L_\alpha$ and 
$\{h\restr_\beta:h\in H\}$ is dense in $G_\beta$
whenever $1\le\beta<\alpha$.
Then $H$ is dense in $L_\alpha$.
\if
the inverse limit 
$$
L_\alpha=\left\{g\in\prod\{G_\beta:1\le\beta<\alpha\}: g\restr_\gamma=\pi^\beta_\gamma(g\restr_\beta) \text{ whenever } 1\le \gamma<\beta<\alpha\right\}
$$ 
of the inverse spectrum $\{G_\alpha, \pi^\beta_\gamma, 1\le\gamma\le\beta<\alpha\}$.
\fi
\end{claim}
\begin{proof}
Let $U$ be an open subset of the product $\prod\{H_\beta:\beta<\alpha\}$
such that $U\cap L_\alpha\not=\emptyset$.
Pick arbitrarily $g\in U\cap L_\alpha$. 
There exist $n\in\omega$, pairwise distinct ordinals $\gamma_0,\gamma_1,\dots, \gamma_n<\alpha$ and an open 
subset $V_i$ of $H_{\gamma_i}$ for every $i\le n$ such that $g(\gamma_i)\in V_i$
for all $i\le n$
and 
\begin{equation}
\label{eq:1}
\left\{h\in \prod\{H_\beta:\beta<\alpha\}:h(\gamma_i)\in V_i \text{ for all }
i\le n\right\}\subseteq U.
\end{equation}
Since $\alpha$ is a limit 
ordinal, $\beta=\max\{\gamma_i:i\le n\}+1<\alpha$.
Note that 
\begin{equation}
\label{eq:2}
W=\left\{h\in \prod\{H_\gamma:\gamma<\beta\}:h(\gamma_i)\in V_i \text{ for all } 
i\le n\right\}
\end{equation}
is an open subset of $\prod\{H_\gamma:\gamma<\beta\}$
and $g\restr_\beta\in W$. Since $g\in L_\alpha$, one has $g\restr_\beta\in G_\beta$
by 
(\ref{def:L}).
It follows that $g\restr_\beta\in W\cap G_\beta\not=\emptyset$.
By the assumption of our claim, there exists some $h\in H$ such that
$h\restr_\beta\in W$. Now from (\ref{eq:1}),  (\ref{eq:2}) and the choice 
of $\beta$ we get $h\in U$.
Thus $h\in H\cap U\not=\emptyset$.
\if
there exist an ordinal $\beta$ with $1\le\beta<\alpha$ 
and an open subset $V$ of $G_\beta$
such that $g(\beta)\in V$ and $\pi_\beta^{-1}(V)\subseteq U$, where
$\pi_\beta: \prod\{G_\beta:1\le\beta<\alpha\}\to G_\beta$ is the projection
on $\beta$th coordinate.
Since $g(\beta)\in V\not=\emptyset$ and $\varphi_\beta(X)$ is dense in $G_\beta$, 
there exists $x\in X$ such that $\varphi_\beta(x)\in V$.
Define $h\in \prod\{G_\beta:1\le\beta<\alpha\}$ by $h(\gamma)=\varphi_\gamma(x)$
whenever $1\le\gamma<\alpha$.
\fi
\end{proof}
\begin{claim}
\label{G:alpha:denseness}
$G_\alpha\subseteq L_\alpha$ and $G_\alpha$ is dense in $L_\alpha$.
\end{claim}
\begin{proof}
Let $h\in G_\alpha$. Then $h=\varphi_\alpha(g)$ for some $g\in G$.
For every ordinal $\beta$ satisfying $1\le\beta<\alpha$ one has
$h\restr_\beta=
\varphi_\beta(g)\in G_\beta$, which yields $h\in L_\alpha$
by (\ref{def:L}).
Thus, $G_\alpha\subseteq L_\alpha$.

Assume that $\beta$ is an ordinal satisfying $1\le\beta<\alpha$.
Let 
$h'\in G_\beta$. Then $h'=\varphi_\beta(g)$ for some $g\in G$.
Now $h=\varphi_\alpha(g)\in G_\alpha$ and
$h\restr _\beta=
\varphi_\beta(g)=h'$. This yields 
$G_\beta\subseteq\{h\restr_\beta:h\in G_\alpha\}$.
The converse inclusion $\{h\restr_\beta:h\in G_\alpha\}\subseteq G_\beta$
is trivial.
This shows that $\{h\restr_\beta:h\in G_\alpha\}= G_\beta$.

Therefore, $G_\alpha$ (taken as $H$) satisfies the 
assumptions of Caim \ref{dense:claim}, so $G_\alpha$ must be dense
in $L_\alpha$ by the conclusion of this claim. 
\end{proof}
\begin{claim}
\label{G:and:L:coincide}
$G_\alpha=L_\alpha$.
\end{claim}
\begin{proof}
The map
$\varphi_\alpha$
is open  
by Claim \ref{phi:are:perfect}
and
Fact 
\ref{quotient:implies:open}. As an open continuous image of a locally compact group $G$, the group $G_\alpha=\varphi_\alpha(G)$ is also locally compact.
Since $L_\alpha$ is a topological group containing $G_\alpha$ (Claim \ref{G:alpha:denseness}), $G_\alpha$ must be closed in $L_\alpha$ 
(Fact \ref{locally:compact:subgroups:are:closed}).
Since $G_\alpha$ is also dense in $L_\alpha$ (Claim \ref{G:alpha:denseness}),
the conclusion of our claim follows.
\end{proof}

We are now ready to define $f_\alpha:S(X)\to G_\alpha$. Let $x\in S(X)$ be 
arbitrary. Since (i$_\beta$) holds for every 
ordinal $\beta$ satisfying $1\le\beta<\alpha$, 
there exists a 
unique $h_x\in L_\alpha$ such that $h_x\restr_\beta=f_\beta(x)$ for all $\beta$ with
$1\le\beta<\alpha$. Now $h_x\in G_\alpha$ by Claim \ref{G:and:L:coincide}, and so 
we can define $f_\alpha(x)$ to be this unique $h_x$. 

Let us check now conditions (i$_\alpha$)--(iv$_\alpha$).
Condition (i$_\alpha$) clearly holds.
Since each $f_\beta$ is a continuous map, so is $f_\alpha$. 
(ii$_\beta$) for 
$1\le \beta<\alpha$
trivially implies 
(ii$_\alpha$).
Similarly, (iii$_\beta$) for 
$1\le \beta<\alpha$
yields 
(iii$_\alpha$).
To check (iv$_\alpha$)
it suffices to show, 
in view of Claim \ref{G:and:L:coincide}, 
that $H=\grp{f_\alpha(S(X))}\subseteq G_\alpha=L_\alpha$
satisfies the assumptions of Claim \ref{dense:claim}.
Indeed, assume $1\le \beta<\alpha$.
Since $\pi^\alpha_\beta$ is a group homomorphism, 
from (i$_\alpha$) one has
$$
\{h\restr_\beta:h\in H\}=\{\pi^\alpha_\beta(h):h\in H\}=
\pi^\alpha_\beta(\grp{f_\alpha(S(X))})=
\grp{\pi^\alpha_\beta(f_\alpha(S(X)))}=
\grp{f_\beta(S(X))},
$$
and the latter set is dense in $G_\beta$ by (iv$_\beta$).

The recursive construction has been complete. 

According to (ii$_\tau$), we have $f_\tau(*)=1_{G_\tau}$.
According to (iv$_\tau$),
$\grp{f_\tau(S(X))}$ is dense in $G_\tau$. 
From Lemma \ref{building:suitable:sets}, 
we conclude that $S=f_\tau(S(X))\setminus \{1_{G_\tau}\}$ is a suitable set for $G_\tau$ such that $S\cup\{1_{G_\tau}\}$ is compact.

Now observe that 
$\ker\varphi_\tau\subseteq \bigcap\{N_\alpha:\alpha<\tau\}\subseteq 
\bigcap\{U_\alpha:\alpha<\tau\}=\{1_G\}$, and hence
$\varphi_\tau:G\to G_\tau$ is an algebraic isomorphism.
Furthermore, 
$\varphi_\tau$ is a perfect map by Claim \ref{phi:are:perfect}.
Finally, note that a one-to-one continuous perfect map is a homeomorphism.
Thus, $G$ and $G_\tau$ are isomorphic as topological groups.
\end{proof}

\noindent
{\bf Proof of Theorem \ref{Hofmann-Morris:theorem}:}
Let $H$ be a locally compact group. Take an open neighbourhood $U$ of the identity
$1_H$ that has a compact closure $\overline{U}$ in $H$.
Then 
$G=\grp{{U}}$ is an open (and thus closed \cite[Chapter II, Theorem 5.5]{HR}) 
subgroup of $H$.
In particular, $\overline{U}\subseteq \overline{G}=G$, and so
$G$ is generated by its open subset $U$ with compact closure (in $G$).
According to Theorem \ref{main:theorem},
$G$ has a suitable set $S$.
Choose $X\subseteq H\setminus G$ such that $\{xG:x\in X\}$ forms a (faithfully indexed) partition 
of $H\setminus G$. 
One can easily check now that
$S\cup X$ is a suitable set for $H$.
\qed

\medskip
\noindent
{\bf Acknowledgement.\/} The author would like to thank the Guest Editor Dikran Dikranjan and the referee for helpful suggestions on the exposition.

\end{document}